\documentclass[12pt]{article}

\usepackage{amsmath,makeidx,amssymb,amscd,amsfonts, amsthm}
\usepackage{graphicx,epsfig,latexsym,bbm,index}

\usepackage{chngcntr}

\usepackage[active]{srcltx}
\usepackage[below]{placeins}

\usepackage{psfrag}
\usepackage{color,epsfig}

\usepackage[dvips, bookmarksopen, colorlinks,
  linkcolor =blue, urlcolor =red, citecolor =red, menucolor =blue]{hyperref}

\usepackage[below]{placeins}
\newcommand\R{\mathbb{R}}

\newcommand\norm[2][]{{\left\lVert#2\right\rVert_{#1}}}

\newtheorem{theo}{Theorem}
\newtheorem{lemma}[theo]{Lemma}

\newtheorem{pr}[theo]{Proposition}

\newtheorem{ex}[theo]{Example}
\newtheorem{assumption}[theo]{Assumption}

\def\.{{\;}}

\begin{document}
\setcounter{footnote}{1}

\title{A note on uniqueness in the identification of a spacewise dependent source and diffusion coefficient for the heat equation}

\author{
A.~De Cezaro %
\thanks{Institute of Mathematics Statistics and Physics,
Federal University of Rio Grande, Av. Italia km 8, 96201-900 Rio
Grande, Brazil (\href{mailto:decezaro@impa.br}{\tt
decezaro@impa.br}).} \ \ and \ \
B. T. Johansson %
\thanks{School of Mathematics, University of Birmingham,
Edgbaston, Birmingham B15 2TT, UK (\href{mailto:b.t.johansson@bham.ac.uk}{\tt
b.t.johansson@bham.ac.uk}).}
}

\date{}

\maketitle
\begin{small}

We investigate uniqueness in the inverse problem of reconstructing
simultaneously a spacewise conductivity function and a heat source
in the parabolic heat equation from the usual conditions of the
direct problem and additional information from a supplementary
temperature measurement at a given single instant of time. In the
multi-dimensional case, we use Carleman estimates for parabolic
equations to obtain a uniqueness result. The given data and the
solution domain are sufficiently smooth such that the required norms
and derivatives of the conductivity, source and solution of the
parabolic heat equation exist and are continuous throughout the
solution domain. These assumptions can be further relaxed using more
involved estimates and techniques but these lengthy details are not
included. Instead, in the special case of the one-dimensional heat
equation, we give an alternative and rather straightforward proof of
uniqueness for the inverse problem, based on integral
representations of the solution together with density results for
solutions of the corresponding adjoint problem. In this case, the
required regularity conditions on the conductivity, source and the
solution of the parabolic heat equation are weakened to classes of
integrable functions.

\end{small}
\section{Introduction} \label{sec:1}

Inverse problems of parameter identification in partial differential
equations have several important applications including thermal
prospection of material and bodies, hydraulic prospecting of soil,
photonic detection of cancer, finding pollution sources, see, for
example, \cite{EngHanNeu96, Isa90, Isa06, IK2000, Ivanchov2003,
Poisson2008, Vessella2008, Yamamoto2009} and references therein. For
inverse problems in general it is important to find and specify
appropriate data such that the set of parameters to be reconstructed
are uniquely identifiable. We shall consider an inverse problem for
the parabolic heat equation, where the additional data is
information about the solution obtained from a spacewise measurement
at the final time. To be more specific, we shall show uniqueness of
the identifiability of a pair of functions $(a(x), f(x))$,
representing the spacewise thermal conductivity and the heat source,
in the parabolic heat equation
\begin{align}\label{eq:model2}
u_t - \nabla \cdot (a(x) \nabla u) & = f(x) \,\, \mbox{in } \ \Omega \times (0,T)\nonumber\\
        u(x,t) & = 0    \,\, \mbox{for} \,\, (x, t) \in \partial \Omega
        \times (0,T)\\
        u(x,0) & = h(x) \,\, \mbox{for }  x \in \Omega,\nonumber
\end{align}
for a given initial temperature $h$, with the additional temperature
measurement $g$ at time $t = T$, i.e.
\begin{align}\label{eq:ad-measurement}
u(x,T) = g(x) \,\, \mbox{for } x \in \Omega\,,\quad T > 0\,.
\end{align}

We assume that $a$ and $f$ are spacewise dependent real-valued
functions and that there exists $\underbar{a} > 0$ such that $ a(x)
\geq \underbar{a}$ for all $x \in \Omega$. This implies that the
operator $L u = - \nabla \cdot (a(x) \nabla u)$ is elliptic. For the
moment we assume that the coefficient, the initial and final
conditions are sufficiently smooth such that there exist a unique
classical solution $u(x,t)$ of the problem \eqref{eq:model2} and
that the required compatibility conditions are satisfied.  For the
precise statement of smoothness and other assumptions to guarantee
the existence of such a classical solution of the direct
problem~(\ref{eq:model2}), see \cite[Theorem 5.2]{LandSolUr-1968}.
Regularity conditions will be further discussed in
Section~\ref{section:local_uniqueness} and in the
Appendix~\ref{appendix}.

There are many contributions on uniqueness for the identifiability of coefficients
in parabolic type equations in the case of lateral
overdetermination. Since the literature on this subject is vast
we suggest the reader to consult \cite{Isa06, IK2000,
KlibanovTimonov2004, Poisson2008, Yamamoto2009} and references
therein.

On the other hand, there are only a few papers related to the
inverse problem of spacewise coefficient identification with given
additional measurements at the final time. Indeed, uniqueness from
final time data for a spacewise dependent heat source was shown in
\cite{Rundel80} and simultaneous determination of a heat source and
initial data from spacewise measurements was investigated in
\cite{JohanssonLesnic2008}. Uniqueness from final time data for the
identification of diffusion coefficients was shown in a recent paper
\cite{Goldman}, for the one-dimensional heat equation. We point out
that the arguments proving uniqueness for the one-dimensional
setting \cite{Goldman} appear not possible to extend to the
multi-dimensional case.

In \cite{Yamamoto2009}, Carleman type estimates were used to prove
uniqueness and stability of a sufficiently smooth diffusion
coefficient in the heat equation, with measurements given at an
intermediate time. However, to the authors' knowledge there are no
results on uniqueness for the reconstruction of both a spacewise
dependent conductivity and heat source. Thus, we shall state and
prove  such a uniqueness result building on Carleman estimates
in~\cite{Yamamoto2009}. For the ease of presentation and to
highlight the usefulness of Carleman estimates for the inverse
problem \eqref{eq:model2}--\eqref{eq:ad-measurement}, we shall
simplify the proof using smoothness assumptions together with some
a priori knowledge of the conductivity close to the boundary of the
solution domain. It is conjectured that these smoothness assumptions
can be removed but we only indicate possible generalizations rather
than give full lengthy and complex details. To convince the reader
that it is possible to have uniqueness also in spaces of integrable
functions, we give a rather straightforward proof of uniqueness in
the case of a one-dimensional solution domain, where the smoothness
assumptions are relaxed and more general. This proof is based on a
recent work~\cite[Section 5]{Goldman} and involves integral
relations obtained from Green's formula for the solution together with some denseness properties of the
adjoint (heat) equation. This technique does not appear possible
though to generalize to higher dimensions.

We point out that once uniqueness is shown for
smooth data, one can use standard approximation techniques together
with stability results for parabolic equations to get a result for
also for non-smooth data.

Note that it is crucial to have only spacewise dependence in the
heat conductivity and source term; there are examples showing
non-uniqueness of the reconstruction of the heat source in the case
of time-dependent sources, see \cite{AlessandriniEscouriaza2008,
Goldman2011, Isakov91}. A simple counterexample to the uniqueness in
the case of a time-dependent heat conductivity is presented in
\cite[Section 5]{Goldman}.

We point out that an interesting related problem that we do not
explore further in this contribution is conditional
stability estimates for the above inverse problem
\eqref{eq:model2}--\eqref{eq:ad-measurement}. In general, for an
inverse problem, in spite of the ill-posedness conditional stability
estimates assure that one can restore the stability of the requested
(physical) parameters provided they are restricted to some class
within an a priori bounded set. The conditional stability is not
only of theoretical interest but is also of importance for the
construction of numerically stable solutions. For example, in
\cite{ChengYamamoto2000} a stability estimate gives convergence
rates for certain Tikhonov regularized solutions. There are several
methods in the literature for proving conditional stability, see
\cite{Isa06, Yamamoto2009}. The method of Carleman estimates is one
possibility for obtaining conditional stability. Thus, it is
possible to obtain such stability based on the results presented in
the present paper. We do not go into the details here, it is
deferred to future work. The reader can further consult \cite{Isa06,
Poisson2008, Yamamoto2009} and the references therein.

For the outline of this paper, we show a uniqueness result for the
above inverse problem using Carleman estimates, see
Section~\ref{section:local_uniqueness} and Theorem 4. In
Section~\ref{section:global_uniqueness}, we address global
uniqueness for the identification of the heat source and
conductivity in the one-dimensional case, using density arguments
for solutions of the corresponding adjoint problem. The arguments in
Section~\ref{section:global_uniqueness} weakens the smoothness
assumptions of the spacewise source and coefficients in the inverse
problem \eqref{eq:model2}--\eqref{eq:ad-measurement}. In the final
section, we draw some conclusions and discuss some possible
generalization of the presented uniqueness results. For the sake of
completeness, in the Appendix~\ref{appendix}, we prove that for a
given class of initial data, the solution of the heat equation at
the final time satisfies the required assumptions stated in
Section~\ref{section:local_uniqueness}.

\paragraph{Notation:} We finish the introduction stating some
notation that we use: The set $\Omega$ is an
open and bounded subset of $\R^n$, with the boundary $\partial \Omega$ being at least
Lipschitz smooth. By $L^p(\Omega)$ for $1\leq p < \infty$, we denote the
usual space of $p$-integrable functions on $\Omega$ with the usual
norm $\norm{\cdot}_{L^p(\Omega)}$. The space $L^\infty(\Omega)$ is
the standard $L^\infty$-space. We denote by $W^{k,p}(\Omega)$ the
standard Sobolev space on $\Omega$ with generalized derivatives of
order $\leq k$ in $L^p(\Omega)$. In particular, for $p=2$ we have
the Hilbert spaces $H^k(\Omega)$. Moreover, since $\partial \Omega$
is Lipschitz, the trace of a function in $H^1(\Omega)$ to the
boundary is well-defined.

Let $T > 0$ be fixed and define the measurable function $u(\cdot,
t)\,:\; (0,T) \longrightarrow X$, where $X$ is a Banach space. We
denote by $C([0,T]; X)$ the space of continuous mappings $u(\cdot,
t)$ with the usual norm and by $L^2((0,T);X)$ the space of functions
such that
\begin{align*}
\int_0^T \norm{u(\cdot, t)}^2_{X} \,dt < \infty\,.
\end{align*}

We also assume that $a$ and $f$ are sufficiently regular spacewise
real-valued functions with $ 0 < \underline{a} \leq a(x)$ for every
$x \in \Omega$. Therefore, the differential operator
\begin{align*}
L u = \nabla (a(x) \nabla u)
\end{align*}
is elliptic. Since, the operator $-L$ generate a contraction
semigroup, there exists a unique solution $u$ of
\eqref{eq:model2} with
\begin{align*}
\norm{u}_1 = \int_0^T\left(\norm{u_t(\cdot,t)}^2_{L^2(\Omega)} +
\norm{u(\cdot,t)}^2_{H^2(\Omega)} \right) \,dt < \infty\,.
\end{align*}

\section{Local uniqueness for the spacewise source and heat
conductivity in (1)--(2)}\label{section:local_uniqueness}

As mentioned, we do not strive to obtain the most general
result, but we only wish to convince the reader that there can
be at most one spacewise dependent heat conductivity coefficient and one
spacewise dependent heat source that together satisfy the given final time data, i.e. that solve the inverse problem \eqref{eq:model2}--\eqref{eq:ad-measurement}.
Therefore, we assume that the given data and the solution domain are sufficiently smooth
such that the required norms and derivatives of the coefficients,
sources and solution of \eqref{eq:model2}--\eqref{eq:ad-measurement}
exist and are continuous throughout the solution domain. For the
precise statement of smoothness and compatibility conditions for the parabolic heat equation, see
\cite[Theorem 5.2]{LandSolUr-1968}. Once uniqueness is shown for
smooth data, one can use standard approximation techniques together
with stability results for parabolic equations to get a result for
non-smooth data. Note that there are now many solvability results
and estimates for parabolic equations with very general
coefficients, see further \cite{LandSolUr-1968, Krylov2007}.

For our proof of uniqueness we use local Carleman estimates and
follow \cite[Section 6]{Yamamoto2009}. To simplify the
presentation further and to avoid cut-off functions and global
Carleman estimates, we assume that the diffusion coefficient is
known in a region $\Omega - \overline{D}$, where $\overline{D}
\subset \subset \Omega$, and $ 0 < |D| < \infty$. In other words,

\begin{assumption}\label{ass:a}
The diffusion coefficient $a(x)$ in\/~\eqref{eq:model2} is
known for every $x \in \Omega \setminus \overline{D}$, where $D
\subset  \Omega$ with $\partial D$ sufficiently smooth and $d(D,
\partial \Omega) > \gamma > 0$.
\end{assumption}
This is a reasonable assumption in applications, since the material
(body) $\Omega$ might be coated or layered and the physical properties of the outer layer is
known. The above assumption forces any two solutions of the inverse problem \eqref{eq:model2}--\eqref{eq:ad-measurement} to be equal in
$\Omega \setminus \overline{D}$, i.e. the difference has compact support
in $\Omega$. More precisely, we have the following lemma:
\begin{lemma}\label{lemma:compact_support}
Let $u = u(a, f_1)$ and $v =v(b, f_2)$ be solutions of
\eqref{eq:model2}--\eqref{eq:ad-measurement}, with spacewise heat
conductivities $a$ and $b$ and spacewise heat sources $f_1$ and $f_2$,
respectively. Assume that $a = b$ in $\Omega - \overline{D}$. Then
the difference $ w = u - v $ is such that $w = 0$ in $\Omega \setminus
\overline{D}$. Moreover $f_1 = f_2$ in $\Omega - \overline{D}$.
\end{lemma}
\begin{proof}
By linearity of \eqref{eq:model2}, $w$ satisfies
\begin{align}\label{eq:model2.1}
w_t - \nabla \cdot (a(x) \nabla w) & =  \nabla \cdot ((a-b) \nabla v) +  (f_1 - f_2)(x) \,\, \mbox{in } \ \Omega \times (0,T)\nonumber\\
        w(x,t) & =  0   \,\, \mbox{for} \,\, (x, t) \in \partial \Omega \times (0,T)\\
        w(x,0) & = w(x,T) = 0 \,\, \mbox{for }  x \in \Omega. \,\nonumber
\end{align}
From the assumption that $a = b$ in $\Omega - \overline {D}$ we have,
\begin{align}\label{eq:model2.1.0}
w_t - \nabla \cdot (a(x) \nabla w) & =  (f_1 - f_2)(x) \,\, \mbox{in } \, \Omega - \overline{D} \times (0,T)\nonumber\\
        w(x,t) & =  0   \,\, \mbox{for} \,\, (x, t) \in \partial \Omega \times (0,T)\\
        w(x,t) & = g_0(x) \,\, \mbox{for} \,\, (x, t) \in \partial D \times (0,T)\nonumber \\
        w(x,0) & = w(x,T) = 0 \,\, \mbox{for }  x \in \Omega.\nonumber
\end{align}
Therefore, using standard parabolic theory we conclude that there
exist only one $w \in C([0,T], H_0^1(\Omega) \times H^2(\Omega)$
solution of \eqref{eq:model2.1.0}.

Let we define the following problems
\begin{align}\label{eq:model2.1.1}
w^{(1)}_t - \nabla \cdot (a(x) \nabla w^{(1)}) & =  (f_1 - f_2)(x) \,\, \mbox{in } \, \Omega - \overline{D} \times (0,T)\nonumber\\
        w^{(1)}(x,t) & =  0   \,\, \mbox{for} \,\, (x, t) \in \partial \Omega \times (0,T)\\
        w^{(1)}(x,t)& = 0 \,\, \mbox{for} \,\, (x, t) \in \partial D \times (0,T)\nonumber \\
        w^{(1)}(x,0) & = w^{(1)}(x,T) = 0 \,\, \mbox{for }  x \in \Omega, \nonumber
\end{align}
Since the source is in $L^2(\Omega)$, we have  from \cite[Theorem
2]{Rundel80} that there exist a unique solution $w^{(1)}$ of
\eqref{eq:model2.1.1}.

Therefore, taking the difference $w^{(2)}=w - w^{(1)}$, we find that
$w^{(2)}$ exists and satisfies
\begin{align}\label{eq:model2.1.2}
w^{(2)}_t - \nabla \cdot (a(x) \nabla w^{(2)}) & =  0 \,\, \mbox{in } \, \Omega - \overline{D} \times (0,T)\nonumber\\
        w^{(2)}(x,t) & =  0   \,\, \mbox{for} \,\, (x, t) \in \partial \Omega \times (0,T)\\
        w^{(2)}(x,t) & = g_0(x,t) \,\, \mbox{for} \,\, (x, t) \in \partial D \times (0,T)\nonumber\\
        w^{(2)}(x,0) & = w^{(2)}(x,T) = 0 \,\, \mbox{for }  x \in \Omega .\nonumber
\end{align}


We start by considering the solution $w^{(1)}$
of~\eqref{eq:model2.1.1}. From~\cite[Theorem 1]{Rundel80} given a
final condition at time $T > 0$ there is at most one spacewise
dependent source giving rise to this final time value for the heat
equation with homogeneous boundary and initial condition. Thus,
since $w^{(1)}$ is zero at $t=T$ and has homogeneous boundary and
initial condition we conclude that $f_1-f_2=0$, this in turn implies
that $w^{(1)}=0$ in $\Omega - \bar{D} \times (0,T)$.

We then consider the solution $w^{(2)}$ of~\eqref{eq:model2.1.2}. If
the boundary data $g_0$ is identically zero it is clear that
$w^{(2)}=0$. Therefore, assume that $g_0$ is not identically zero.
Without loss of generality, we can assume that $g_0(x,t) > 0$ at a
point $(x,t)$ of the boundary $\partial D \times (0,T)$ and since
$g_0$ is at least continuous it follows that $g_0(x,t) > 0$ in $E
\times (t_1,t_2)$, with $E$ being a surface patch of $\partial D$.
This in turn implies that $w^{(2)}(x,t)> 0$ for $(x,t)$ sufficiently
close to this open set. Applying \cite[Theorem 9.2]{Isa90} it
follows that $w^{(2)}(x,s) > 0$ for every $t < s \leq T$
contradicting the given final condition $w^{(2)}(x,T)=0$. Therefore,
$g_0=0$ implying that $w^{(2)}=0$.

Thus, since $w=w^{(1)} + w^{(2)}$ we have $w=0$ and therefore the assertion of the lemma is proved.
\end{proof}
Note that, if the measurements are given at an intermediate time $0 <
t_0 < T$, the conclusions of Lemma~\ref{lemma:compact_support} may
not be true.

\subsection{Some additional assumptions and admissible solutions}
In order to proceed to the next step in the proof of uniqueness of
the spacewise dependent pair $(a,f)$ satisfying \eqref{eq:model2}--\eqref{eq:ad-measurement} we need some additional assumptions.
We assume that the admissible set of unknown
elements is
\begin{align}\label{ad_coefficients}
\mathcal{A}: = \{(a,f) \in C^2(\overline{\Omega})\,:\, a(x) >
\underline{a} > 0\,, x \in \Omega\,, \quad
\norm{a}_{C^2(\overline{\Omega})} +
\norm{f}_{C^2(\overline{\Omega})} \leq M \}\,,
\end{align}
where $M > 0$ is an arbitrary fixed constant. Moreover, we assume
that $g$ and $h$ are smooth enough such that we can take the
$t$-derivatives of $w = u - v$ for the time points needed, see
\cite{LandSolUr-1968} for the details on the required smoothness
assumptions on $g$ and $h$ for this to be the case. Note that in
Section~\ref{section:global_uniqueness} and in
Appendix~\ref{appendix} we will provide more details on the
regularity of the parameters and the input data. Moreover, we also
assume that $\partial \Omega$ is smooth enough such that $\partial
\Omega \subset \{x_1 = 0\}$. In fact, the boundary can be covered by
a finite number of surface patches mapping to this region and it is
therefore enough to consider the estimates in this half-space. The
general results can be obtained in standard way for partial
differential equations by using a partition of unity argument.

We remark that since the coefficients and the source term are
time-independent, the solution of \eqref{eq:model2} is analytic in time
\cite{Isa06, Vessella2008}. Moreover, the compact support of $w$ in $\Omega$ guaranteed by
Lemma~\ref{lemma:compact_support} implies the follows boundary
conditions
\begin{align}\label{eq:boundary_w}
w_t(x,t)  & = 0 \quad (x,t) \in \partial \Omega \times (0,T)\,, \nonumber\\
\frac{\partial }{\partial \eta} w_t(x,t)  & = 0 \quad (x,t) \in
\partial \Omega \times (0,T)\,, \\
\frac{\partial }{\partial \tau} w_t(x,t)  & = 0 \quad (x,t) \in
\partial \Omega \times (0,T)\,, \nonumber
\end{align}
where $\eta$ is the outer normal vector and $\tau$ is the tangential
normal vector  at $\partial \Omega$. Moreover, $\frac{\partial
}{\partial \eta} w_t(x,t) = \eta(x) \cdot \nabla w_t(x,\cdot)$ and $
\frac{\partial }{\partial \tau} w_t(x,t) = \tau (x) \cdot \nabla
w_t(x, \cdot)$ are the normal and tangential derivatives of $w_t$
for $x \in \partial \Omega$.

We shall present
a uniqueness result for the inverse problem \eqref{eq:model2}--\eqref{eq:ad-measurement} based on local Carleman estimates.
A crucial step for such estimates is the construction of suitable weight functions,
$\varphi(x,t)$ and $\beta(x,t)$. For a
given domain $D$ we can apply the arguments in \cite[Section
5.1]{Yamamoto2009} or the arguments in \cite[Section
2]{Poisson2008}, with a special choice of a function $d(x)$ such
that
\begin{align} \label{eq:psi}
\varphi(x,t) = e^{\lambda \beta(x,t)}\,
\end{align}
and $\beta(x,t)=d(x)+ e(t)$, where  $\max_{t\in [0,T]}
\varphi(x,t) = \varphi(x,T)$. Moreover, we can construct
$\varphi$ and $\beta$ with
\begin{align}\label{eq:weight_functions}
Q:= \{(x,t)\,:\, x_1 > 0\,,\, \varphi(x,t) > e^{-\lambda \delta}\} =
\{(x,t)\,:\, x_1 > 0\,,\, \beta(x,t) >- \delta\}\,,
\end{align}
and $Q \cap \{t = T\} = D$. In \eqref{eq:psi} and
\eqref{eq:weight_functions}, $\lambda, \delta > 0$ are some fixed
constants being sufficiently large.

Let we give a very simple example for choice $d(x)$
and $e(t)$ form \cite[Section 6]{Yamamoto2009}.
%
\begin{ex} Set $x'= (x_2, \cdot, x_n)$ and $x =
(x_1, x') \in \R^n$. Let we assume that $D = D(\delta):=\{(x,x') ; 0
< x_1 < - |x'|^2/ \gamma +\delta/ \gamma \}$ and moreover
$D(4\delta)\subset \Omega$, for some $\gamma > 0$ and $\delta > 0$.
Define $d(x) := - \gamma x_1 - |x'|^2$ and $e(t):= - (t- T)^2$.
Therefore, we have by \eqref{eq:psi} that $\max_{t\in [0,T]}
\varphi(x,t) = \varphi(x,T)$ and that $Q = Q(\delta)$ satisfies $Q
\cap \{t = T\} = D$.
\end{ex}

As mentioned earlier, since we have spacewise dependent
coefficients, the solution of \eqref{eq:model2} is analytic in time.
Therefore, the solution can be extended beyond the final time $T$, see
\cite[Section 3]{Vessella2008} and $T$ can therefore be considered as an interior
point, as is required in the derivation of the results in \cite[Section 6]{Yamamoto2009}.

We remark that one can try to identify the most general set of conductivity coefficients having minimal
regularity assumptions together with minimal regularity of the boundary of $\Omega$ for which the weight functions $\varphi$ and $\beta$ do exist. However, as pointed out earlier this is not the main
aim of this study, we shall only present a proof in the case of smooth and regular solutions and domains. For a discussion about more general function spaces and domains for which this derivation can hold true, see \cite{Poisson2008, Yamamoto2009} and references therein.

We also remark that since $w$ has compact support (see
Lemma~\ref{lemma:compact_support}), we do not need to use cut-off
functions. Therefore, many of the calculations and terms in \cite[Theorem
6.1]{Yamamoto2009} can be dropped and the steps in the proof become
easier to follow. However, since we have two unknowns, some
challenges and adjustments do remain, and we present the steps below.

\subsection{A uniqueness proof}

Let $u = u(a, f_1)$ and $v =v(b, f_2)$ be solutions of
\eqref{eq:model2}--\eqref{eq:ad-measurement}, with spacewise heat
diffusions $a$ and $b$ and spacewise heat sources $f_1$ and $f_2$,
respectively, and put $ w = u - v $. Denote by $ z:= w_t$; a well-defined quantity due to the smoothness assumptions above and since we only work with classical solutions. Given the analyticity of $w$ in time and the assumption
that the unknown pair $(a,f)$ is time independent, it follows that $z$ satisfies
\begin{align}\label{eq:model2.3}
z_t - \nabla \cdot (a(x) \nabla z) & =  \nabla \cdot ((a-b) \nabla
v_t)
\end{align}
with homogeneous boundary, initial and final conditions. Then,
\cite[Theorem 3.2]{Yamamoto2009} implies
\begin{align}\label{eq:carleman}
\int_Q \left( \frac{1}{s} \sum_{i,j =1}^n | \partial_i \partial_j
z|^2 + s |\nabla z|^2 + s^2|z|^2 \right) e^{2 s \varphi} \,dxdt \leq C
\int_Q |\nabla \cdot ((a - b) \nabla v_t)|^2 e^{2 s \varphi} \,dxdt
\end{align}
for all sufficiently large $s > 0$. Note that the integral over the boundary in
\cite[Theorem 3.2]{Yamamoto2009} is identically zero due to
\eqref{eq:boundary_w}.

Now we shall derive an estimate of the right-hand side of
\eqref{eq:carleman}. A direct calculation of $\nabla \cdot ((a -
b) \nabla v_t)$
implies that
\begin{align}\label{eq:direct-calculation}
 |\nabla \cdot ((a - b) \nabla v_t)|^2 \leq (|\nabla( a- b)|^2 +
|a-b|^2) (| \nabla v_t|^2 + |\triangle v_t|^2)\,.
\end{align}
Hence,
\begin{equation}\label{est1}
\int_Q
|\nabla \cdot ((a - b) \nabla v_t)|^2 e^{2 s \varphi} \,dxdt \leq C
(|| \nabla v_t||^2_\infty + ||\triangle v_t||^2_\infty) \int_Q
(|\nabla( a- b)|^2 + |a-b|^2) e^{2 s \varphi} \,dxdt\,.
\end{equation}
Since $u, v$ are assumed smooth enough (see \cite{LandSolUr-1968} for necessary requirement on the data) we have
that $(|| \nabla v_t||^2_\infty + ||\triangle v_t||^2_\infty) <
\infty$. Therefore, $C = (|| \nabla v_t||^2_\infty + ||\triangle
v_t||^2_\infty) $ is finite.

Combining (\ref{eq:carleman}) and (\ref{est1}) imply
\begin{align}\label{eq:carleman1}
\int_Q \left(\sum_{i,j =1}^n | \partial_i \partial_j z|^2 + s^2
|\nabla z|^2 + s^3|z|^2 \right) e^{2 s \varphi} \,dxdt \leq  C
\int_Q s (|\nabla( a- b)|^2 + |a-b|^2) e^{2 s \varphi} \,dxdt 
\end{align}
for all sufficiently large $s > 0$. 

Given the final time measurement \eqref{eq:ad-measurement},
equation~\eqref{eq:model2.1} gives, at $t = T$,
\begin{align}\label{eq:eliptic}
\nabla \cdot ((a-b) \nabla g(x)) &  =  w_t(x,T) 
+ (f_1 - f_2)(x)\\
\nabla(\nabla \cdot ((a-b) \nabla g(x))) &  = \nabla w_t(x,T) 
+ \nabla(f_1 - f_2)(x)\,. \nonumber
\end{align}
%
%
It follows from a straightforward manipulation of \eqref{eq:eliptic} together with integration that
%
\begin{align}\label{eq:1-l}
\int_D (|f_1 - f_2|^2 & + |\nabla(f_1 - f_2)|^2) e^{2 s
\varphi(x,T)}\,dx \nonumber\\
& + \int_D (|\nabla \cdot ((a-b)) \nabla g(x)|^2 + |\nabla(\nabla
\cdot ((a-b) \nabla g(x)))|^2) e^{2 s \varphi(x,T)}\,dx \nonumber\\
& \leq \int_D (|w_t(x,T)|^2 + |\nabla w_t(x,T)|^2 ) e^{2 s
\varphi(x,T)}\,dx \\
& + 2 \int_D |(\nabla \cdot ((a-b) \nabla g(x)) (f_2 - f_1)(x))|
e^{2 s \varphi(x,T)}\,dx \nonumber\\
&+ 2 \int_D |(\nabla(\nabla \cdot ((a-b) \nabla g(x)))\nabla(f_2 -
f_1)(x))| e^{2 s \varphi(x,T)}\,dx. \nonumber
\end{align}

We put
$$ A = 2 \int_D |(\nabla \cdot ((a-b) \nabla g(x)) (f_2 - f_1)(x))|
e^{2 s \varphi(x,T)}\,dx$$
and
$$ B = 2 \int_D |(\nabla(\nabla \cdot ((a-b) \nabla g(x)))\nabla(f_2 -
f_1)(x))| e^{2 s \varphi(x,T)}\,dx\,.$$

Note that
\begin{align}\label{eq:norm}
\norm{\partial_j w_t e^{s \varphi}}_{H^1(Q)}^2 \leq \int_Q
\left(\sum_{i,j = 1}^n |\partial_i \partial_j w_t|^2 + s^2|\nabla
w_t|^2 + s^3|w_t|^2 \right)e^{2 s \varphi } \,dx dt .
\end{align}
Applying the trace theorem \cite{Ada75} in $Q$ and noting that $Q
\cap \{t=T\} = D$ we have
\begin{align}\label{eq:norm1}
\int_D(|w_t(x,T)|^2 + |\nabla w_t(x,T)|^2)e^{2s \varphi(x,T)} \,dx
\leq  C \sum_{j=0}^1\norm{\partial_j w_t e^{s
\varphi}}_{H^1(Q)}^2\,.
\end{align}

Now, \eqref{eq:carleman1}, \eqref{eq:1-l}, \eqref{eq:norm} and
\eqref{eq:norm1} yield
\begin{align}\label{eq:2}
\int_D (|f_1 - f_2|^2 & + |\nabla(f_1 - f_2)|^2) e^{2 s
\varphi(x,T)}\,dx \nonumber\\
& + \int_D (|\nabla \cdot ((a-b)) \nabla g(x)|^2 + |\nabla(\nabla
\cdot ((a-b) \nabla g(x)))|^2) e^{2 s \varphi(x,T)}\,dx \nonumber\\
& \leq C \int_Q s(|\nabla(a-b)|^2 + |a-b|^2 ) e^{2 s
\varphi(x,t)}\,dxdt + (A+B)\,. 
\end{align}
%

The next step in the proof is an estimate of the second integral
in the left hand-side of equation~\eqref{eq:2}.
\begin{lemma}\cite[Lemma
6.1]{Yamamoto2009}\label{lemma:Yamamoto_6.1} Let   $(a-b) \in
H^2(\overline{\Omega})$ such that $|a-b| = |\nabla(a-b)| = 0$ on
$\partial \Omega$. Assume that
\begin{align}\label{eq:assumption_Lemma_6.1}
\gamma \partial_1 g(x) + 2 \sum_{j=2}^n (\partial_j g)(x) & \leq 0
\quad x \in \overline{\Omega}\,,\\
\partial_1 g(x) & > 0\,,\quad x \in \partial \Omega.\nonumber
\end{align}
Then there exists a constant $C > 0$ such that
\begin{align}\label{eq:6.22}
\int_\Omega s^2(|\nabla(a-b)|^2 & + |a-b|^2) e^{2 s \varphi(x,T)} \,dx \nonumber \\
& \leq C \int_\Omega (|\nabla \cdot ((a-b)) \nabla g(x)|^2 +
|\nabla(\nabla \cdot ((a-b) \nabla g(x)))|^2) e^{2 s
\varphi(x,T)}\,dx
\end{align}
for all sufficiently large $s > 0$.
\end{lemma}
Note that, since we assume that $a = b$ in $\Omega -
\overline{D}$, the estimate \eqref{eq:6.22} holds over $D$.
This in turn using equation~\eqref{eq:2} and
Lemma~\ref{lemma:Yamamoto_6.1} give

\begin{align}\label{eq:3}
\int_D (|f_1 - f_2|^2 & + |\nabla(f_1 - f_2)|^2) e^{2 s
\varphi(x,T)}\,dx  + \int_D s^2(|\nabla(a-b)|^2 + |a-b|^2) e^{2 s \varphi(x,T)} \,dx \nonumber\\
& \leq C \int_Q s(|\nabla(a-b)|^2 + |a-b|^2 ) e^{2 s
\varphi(x,t)}\,dxdt  + (A+B)
\end{align}
for all sufficiently large $s>0$.
Since $\varphi(x,t) \leq \varphi(x,T)\,, x \in
(\overline{\Omega})\,,\quad 0 \leq t \leq T$, we have
\begin{align}\label{eq:4}
\int_D (|f_1 - f_2|^2 & + |\nabla(f_1 - f_2)|^2) e^{2 s
\varphi(x,T)}\,dx  + \int_D s^2(|\nabla(a-b)|^2 + |a-b|^2) e^{2 s \varphi(x,T)} \,dx \nonumber\\
& \leq C T \int_D s(|\nabla(a-b)|^2 + |a-b|^2 ) e^{2 s
\varphi(x,T)}\,dxdt + (A+B)
\end{align}
for all sufficiently large $s>0$.

The final step in the proof of uniqueness is to obtain an estimate of the quantity $A + B$ in~(\ref{eq:4}) in terms of
the coefficients of the first equation in (\ref{eq:model2.1}). We assume that
\begin{align}\label{eq:smoothness-g}
\norm{\nabla (\triangle g(x))}_\infty + \norm{\triangle g(x)}_\infty
+ \norm{\nabla g}_\infty < \infty\,.
\end{align}

Using the Young inequality with $\varepsilon>0$ we have
\begin{align*}
A = 2 \int_D |(\nabla \cdot ((a-b) \nabla g(x)) (f_2 - f_1)(x))|
e^{2 s \varphi(x,T)}\,dx & \leq \frac{2}{\varepsilon} \int_D |(\nabla
\cdot ((a-b) \nabla g(x))|^2 e^{2 s \varphi(x,T)}\,dx \\
& + \varepsilon \int_D | (f_2 - f_1)(x))|^2 e^{2 s \varphi(x,T)}\,dx
\end{align*}
and
\begin{align*}
 B = 2 \int_D |(\nabla(\nabla \cdot ((a-b) \nabla g(x)))\nabla(f_2
- f_1)(x))| e^{2 s \varphi(x,T)}\,dx & \leq \frac{2}{\varepsilon}
\int_D |(\nabla(\nabla \cdot ((a-b) \nabla g(x)))|^2  e^{2 s
\varphi(x,T)}\,dx \\
& + \varepsilon \int_D |\nabla(f_2 - f_1)(x))|^2 e^{2 s
\varphi(x,T)}\,dx.
\end{align*}

Using the similar estimate as before for $ |(\nabla \cdot
((a-b) \nabla g(x))|$, see (14), in combination with assumption \eqref{eq:smoothness-g}
%
we have
\begin{align*}
A <  \frac{C}{\varepsilon}\int_D (|\nabla(a-b) |^2 + |a-b|^2) e^{2 s
\varphi(x,T)}\,dx  + \varepsilon \int_D | (f_2 - f_1)(x))|^2 e^{2 s
\varphi(x,T)}\,dx.
\end{align*}
%
%
%
Similarly, a direct calculation of $ (\nabla(\nabla \cdot ((a-b) \nabla
g(x)))$ 
using the assumption \eqref{eq:smoothness-g} imply
%
\begin{align*}
\int_D |(\nabla(\nabla \cdot ((a-b) \nabla g(x)))|^2  e^{2 s
\varphi(x,T)}\,dx \leq C \int_D(|\triangle (a-b)|^2 + |\nabla (a-b)|^2
+ |a-b|^2)e^{2 s \varphi(x,T)}\,dx.
\end{align*}

We then choose $ 0 < \varepsilon < 1/2$. Using the above estimates
for $A$ and $B$ in (\ref{eq:4}) and put together this give
\begin{align}\label{eq:5}
\int_D (|f_1 - f_2|^2 & + |\nabla(f_1 - f_2)|^2) e^{2 s
\varphi(x,T)}\,dx  + \int_D s^2(|\nabla(a-b)|^2 + |a-b|^2) e^{2 s \varphi(x,T)} \,dx \nonumber\\
& \leq C T \int_D [(s + 1)(|\nabla(a-b)|^2 + |a-b|^2 ) + |\triangle
(a-b)|^2] e^{2 s \varphi(x,T)}\,dxdt\,
\end{align}
for all sufficiently large $s>0$.

For $s > 0$ sufficiently large the terms in the right-hand side can
be absorbed by the corresponding terms in left hand-side. Indeed,
by assumption, $|\triangle
(a-b)|^2 \leq M^2$ and therefore it is enough to choose
$s>0$, such that $ s^2(|\nabla(a-b)|^2 + |a-b|^2) > CT
(s+1)(|\nabla(a-b)|^2 + |a-b|^2) + M^2$. Thus, with such a choice of $s>0$, we conclude that the second
integral in the left-hand side is identically zero, i.e. $a=b$ also
in $D$. The inequality~\eqref{eq:5} then implies that also the sources are equal, i.e.
$f_1=f_2$. Since we concluded that the coefficients are equal one
can alternatively use \cite[Theorem 1]{Rundel80} to obtain
uniqueness of the sources.

The obtained results can be summarised as follows.
\begin{theo}\label{theo:uniqueness}
Let the spacewise conductivity coefficient and heat source $(a,f) \in \mathcal{A}$, where $\mathcal{A}$ is given by~\eqref{ad_coefficients}, with the coefficient $a$
satisfying Assumption~\ref{ass:a}. Moreover, assume that the initial condition
$h$ and the final time measurement $g$ are regular enough such that the
corresponding solution $u(x,t)$ of \eqref{eq:model2} satisfies
$$
\norm{\nabla u_t}_{L^\infty((0,T) \times \Omega)} + \norm{\triangle
u_t}_{L^\infty((0,T) \times \Omega)} < M.
$$
Moreover, assume that the
final time condition $g$ satisfies the
Assumption~\eqref{eq:assumption_Lemma_6.1} in
Lemma~\ref{lemma:Yamamoto_6.1} and \eqref{eq:smoothness-g}. Then
the inverse problem of identifying $\{a(x), f(x), u(a,f)\}$ in the
heat equation \eqref{eq:model2} for a given additional measurement
$g(x) = u(x,T)$ has a unique solution.
\end{theo}

In the Appendix we shall prove that there exist conductivities and
sources which can generate a final time value satisfying \eqref{eq:assumption_Lemma_6.1}, i.e. the
inverse problem \eqref{eq:model2}--\eqref{eq:ad-measurement} with \eqref{eq:assumption_Lemma_6.1} imposed will
have a solution for some data (and this solution is unique according to
the above theorem).

We could now dwell into lengthy calculations on the uniqueness of
the inverse problem \eqref{eq:model2}--\eqref{eq:ad-measurement} under less regularity assumptions on the
coefficient and data. There are indeed many generalizations of
Carleman estimates for parabolic equations that one could potentially
use in this case similar to the arguments in \cite{Poisson2008,
Poisson_Carleman2008, Vessella2008, Yamamoto2009}. However, we
prefer not to enter into these technicalities.

Let us though briefly discuss some alternative estimates in the
above proof of uniqueness that could potentially reduce the imposed
smootheness assumptions. We used the
estimate~\eqref{eq:direct-calculation} to obtain an upper bound of
the right-hand side of \eqref{eq:carleman}. However, there are
alternative estimates that can be used. For example, in
\cite[Subsection 3.2]{Poisson2008} the right-hand side of
\eqref{eq:carleman} is estimated by
\begin{align*}
\int_Q |\nabla v_t|^2|a-b|^2 e^{-2 s \varphi}\,dx dt\,.
\end{align*}
Note that, with the above estimate, equation~\eqref{eq:carleman1}
follows with less assumptions on the regularity of the coefficients
as well as of \eqref{eq:model2}. Moreover, in \cite[Theorem
3.1]{Poisson2008} a Carleman estimate (for $L^\infty(\Omega)$
coefficients) relating the $L^2$-norm of the coefficients with the
right-hand side of equation~\eqref{eq:eliptic} is given. In other
words, \cite[Thoerem 3.1]{Poisson2008} implies that one do not need to
differentiate as high in time as we did in the second equation
of~\eqref{eq:eliptic}. These two factors together would, most likely,
improve and weaken the smoothness assumptions in
Theorem~\ref{theo:uniqueness}.

We point out that the Assumption~1 was imposed merely for technical
reasons to simplify the presentation of the Carleman estimates and
to avoid cut-off functions. To convince the reader that this
assumption can be removed and to further highlight that our
assumptions are far from the most general ones, we discuss in the
next section the uniqueness for a more general class of spacewise
parameters in the one-dimensional parabolic heat equation using a
completely different technique, which does not lend itself to higher
dimensions though.


\section{Uniqueness of the spacewise heat conductivity and heat source in a one-dimensional heat
equation}\label{section:global_uniqueness}

In this section we show the uniqueness under weaker smoothness
assumptions  compared with Section~\ref{section:local_uniqueness} of
the identification of the spacewise pair of coefficient and source
$(a,f)$ in \eqref{eq:model2}, with additional measurements given by
\eqref{eq:ad-measurement}. In particular, the Assumption~1 will not
be used. However, we only consider the case with $\Omega = (0,L)
\subset \R$, i.e. the spacewise solution domain in
\eqref{eq:model2}--\eqref{eq:ad-measurement}  is one-dimensional in
space.

The derivation of the uniqueness result is based on a completely
different technique than Carleman estimates. Indeed, the technique
is based on results that relate the uniqueness of the inverse
identification problem to the density in certain function spaces of solutions of the
corresponding adjoint problem. With this approach the assumptions
on the smoothness of the unknown pair $(a,f)$ is reduced to a
more general class. Moreover, assumptions on the smoothness of the
input and measured data are determined only by extracting a certain
differential dependence and that there exists a solution
for the corresponding direct problem.

We start by assuming the following
regularity conditions for the parameter, the source, the initial
condition  and the measured data in the inverse problems~\eqref{eq:model2}--\eqref{eq:ad-measurement}; compare with
the assumptions in Section~\ref{section:local_uniqueness}.
\begin{assumption}\label{ass:1}
We assume that the heat conductivity $a \in L^\infty((0,L))$ and
that there exists $\underbar{a} > 0$ such that $ a(x) \geq \underbar{a}$
for all $x \in (0,L)$. Moreover, it is assumed that $a(\cdot)$ is continuous in
$[0,\varepsilon)$ and in $(L - \varepsilon, L]$ for any fixed
$\varepsilon > 0$, the heat source $f\in L^2([0,L])$, the initial
temperature $h \in H^1([0,L])$, and that the additional final time temperature
measurement $g\in H^1([0,L])$.
\end{assumption}

Since $h, g \in H^1([0,L])$ we further need to assume that the following
matching conditions (compatibility) are satisfied in $x=0$ and $x=L$
\begin{align}\label{eq:matching-conditions}
- (a(0) h_x(0))_x & = f(0)\,,\nonumber\\
- (a(L) h_x(L))_x & = f(L)\,,\\
- (a(0) g_x(0))_x & = f(0)\,,\nonumber\\
- (a(L) g_x(L))_x & = f(L)\,.\nonumber
\end{align}

For the existence and regularity of a solution for the corresponding
direct problem with the pair $(a(x), f(x))$
satisfying the conditions on regularity stated in
Assumption~\ref{ass:1}, we have:
\begin{lemma}\label{lemma:ex-est-DP}
Let the Assumption~\ref{ass:1} and the matching conditions
\eqref{eq:matching-conditions} hold. Then, there exists a unique
solution $u(x,t) \in H^{2,1}([0,L]\times [0,T])$ of the
boundary value problem \eqref{eq:model2}. Moreover, there exists a constant $M_0$
such that $\norm{u}_{C([0,L]\times[0,T])} \leq M_0$.
\end{lemma}
\begin{proof}
The existence and uniqueness follows directly from classical
results on parabolic partial differential equations, see for example
\cite{LM-vol1, LandSolUr-1968}. Now, from the Sobolev embedding Theorem
\cite{LandSolUr-1968, LM-vol1, Ada75} we have that $u(x,t) \in
C([0,L]\times [0,T])$. The uniform boundedness follows from the
maximum principle for parabolic equations \cite{LandSolUr-1968} together with the assumed
 smoothness of the boundary, initial and final data.
\end{proof}

\subsection{Uniqueness of a solution of the inverse problem: 1-d case}

The steps for proving uniqueness of the identification of
the pair of parameters $\{a(x), f(x)\}$ for given initial and final
data in \eqref{eq:model2}--\eqref{eq:ad-measurement} are outlined below:

Assume that $u= u(a, f_1)$ and $v = u(b, f_2)$ are two solutions of
\eqref{eq:model2} with additional data \eqref{eq:ad-measurement}.
As in the previous section let $ w = u - v$. Then, $w$ satisfies
%
%
\begin{align}\label{eq:model2.1_1}
w_t - (a(x) w_x)_x & =   ([a(x) - b(x)]
 v_x)_x + (f_1(x) - f_2(x)) \,\, \mbox{in } \ (0,L) \times (0,T)
\end{align}
with homogeneous initial, boundary and final conditions.

For the proof uniqueness in the one-dimensional
case we shall invoke the
adjoint problem of \eqref{eq:model2.1_1}, that reads as
\begin{align}\label{eq:model2.Adj}
\psi_t + (a(x) \psi_x)_x & = 0 \,\, \mbox{in }
\ (0,L) \times (0,T)\nonumber\\
        \psi(0,t)=\psi(L,t)& = 0     \,\, \mbox{for} \,\, t \in (0,T)\\
     \psi(x,0) & =  0 \,\, \mbox{for }  x \in (0,L)
     \,\nonumber\\
   \psi(x,T) & =  \mu(x) \,\, \mbox{for }  x \in (0,L),
\nonumber
\end{align}
where $\mu(x)$ is an arbitrary function in $C_0^2[0,L]$.

For properties of solutions of the adjoint equation
\eqref{eq:model2.Adj} we have:

\begin{lemma}\label{lemma:prop-adj}
Let the Assumption~\ref{ass:1} hold.
\begin{itemize}
\item[i)] For any function $\mu(x) \in C_0^2[0,L]$, there exists a unique solution $\psi(x,t; \upsilon) \in C^1((0,T);
C^2(0,L)) \cap C([0,L] \times [0,T])$ of \eqref{eq:model2.Adj}.

\item[ii)] For any function $\mu(x) \in C_0^2[0,L]$, the following relation holds
\begin{align}\label{eq:1}
\int_0^T\int_0^L \psi(x,t; \mu(x)) F(x,t) \,dxdt = 0,
\end{align}
where $w$ is a solution to \eqref{eq:model2.1_1} with right-hand side  $F(x,t)$.

\item[iii)]
For $\mu(x)$ ranging over the space $C_0^2[0,L]$, the corresponding
range of $\psi(x,t; \mu(x))|_{t = \tau}$ is everywhere dense in
$L^2[0,L]$ for any time $t=\tau$, $0\leq \tau
\leq T$. 

\item[iv)] Given that
$$
\int_0^T \int_0^L \psi(x,t; \mu(x)) \Phi(x,t) \,dxdt = 0
$$
for $\mu(x)$ ranging over the space $C_0^2[0,L]$, then
$$
\Phi(x,T)=0.
$$

\end{itemize}
\end{lemma}

\begin{proof}
Item i) is a well-known result for parabolic equations, see, for
example, \cite{LM-vol1, LandSolUr-1968}. \\
Item ii) follows immediately  by multiplication of
\eqref{eq:model2.3} by $\psi$  and integration by parts.\\
Item iii) and Item iv) are consequences of \cite[Lemma 2-3, pg 318]{Goldman}.\\
\end{proof}

We now have the required results in order to prove the main step in the uniqueness argument for the conductivity function $a(x)$ by adding some additional assumptions on the final data in
\eqref{eq:model2}--\eqref{eq:ad-measurement}.
\begin{theo}\label{theo:uniqueness-a}
Let the Assumption~\ref{ass:1} and the matching condition
\eqref{eq:matching-conditions} hold. Moreover, assume that $|g_x(x)| > 0
$ for all $0 \leq x \leq L$. Then the inverse problem
\eqref{eq:model2}--\eqref{eq:ad-measurement} has a unique solution
$\{u, a, f\}$ with the
conductivity $a \in L^\infty(0,L)$, the heat source $f \in L^2(0,L)$,  and temperature $u$, with
$\norm{u}_1 < \infty$.
\end{theo}
\begin{proof}
Since we have time-independent coefficients and source, it follows that the solution $v$ to \eqref{eq:model2} with additional data \eqref{eq:ad-measurement} has derivatives of all orders with respect to $t$ \cite{Isa06, Vessella2008} and this in turn implies that $w$ has derivatives of all orders with respect to $t$. Moreover, due to the Sobolev imbedding theorem, we can assume that $w$ in~(\ref{eq:model2.1_1}) is at least continuous; for simplicity we assume that pointwise evaluation makes sense for the coefficient and source.

Note that at $x=0$ and at $x=L$ using the assumptions and matching
condition \eqref{eq:matching-conditions} imply that
\begin{align}\label{eq:a-boundary}
a(x) = b(x) \quad \mbox{ for } x=0 \mbox{ and } x=L\,.
\end{align}

In order to prove that $a(x) = b(x)$ for $ 0 < x < L$ we apply
Lemma~\ref{lemma:prop-adj} Item ii) in combination with iv) in \eqref{eq:1} to get
$$
[(a(x) - b(x)) (v(x,T))_x]_x + (f_1(x) + f_2(x)) = 0
$$
for $0\leq x \leq L$. Using this in the first equation in (\ref{eq:model2.1_1}) in combination with $w(x,T)=0$ for every $x$ in $[0,L]$, we conclude that $w_t(x,T)=0$.

Define $z = w_t$. Similar to the previous section we
have that $z$ satisfies
%
\begin{align}\label{eq:model2.3}
z_t - (a(x) z_x)_x & =  ([a(x) - b(x)]
( v_x )_t)_x \,\, \mbox{in } \ (0,L) \times
(0,T)\nonumber\\
        z(0,t)=z(L,t)& = 0     \,\, \mbox{for} \,\, t \in (0,T)\\
     z(x,0) & =  \theta_1(x) \,\, \mbox{for }  x \in (0,L)
     \,\nonumber\\
     z(x,T) & =  0 \,\, \mbox{for }  x \in (0,L).
     \,\nonumber
\end{align}
Splitting this problem into two, one with zero right-hand side and with  initial condition $\theta_1$ and one with the given right-hand side and zero initial condition, following the proof of Lemma 2 one can conclude that the solution to the first one is identically zero, i.e. $\theta_1(x)=0$. Therefore, since $z$ satisfies a problem of the same kind as $w$ we can again apply
Lemma~\ref{lemma:prop-adj} Item ii) in combination with iv) in \eqref{eq:1}
to conclude that
$$
[(a(x) - b(x)) (v_t(x,T))_x]_x = 0.
$$
Using this in the first equation in (\ref{eq:model2.3}) in combination with $z(x,T)=0$ for every $x$ in $[0,L]$, we conclude that $z_t(x,T)=0$, i.e. $w_{tt}(x,T)=0$. Continuing this, putting $z_1=z_t$ and deriving the problem for $z_1$ and applying the similar reasoning, i.e. Lemma~\ref{lemma:prop-adj} Item ii) in combination with iv), we find that $w_{ttt}(x,T)=0$. Further continuing this it is possible to prove that $(\partial_t^{(k)} w)(x,T) = 0$ for $k=0,1,2,\ldots$ and the same holds at $t=0$. From this and strong unique continuation results for parabolic equations, we conclude that $w(x,t)=0$ for $[0,L] \times [0,T]$. This in particular implies that $z=0$ in $[0,L] \times [0,T]$. From the first equation in (\ref{eq:model2.3}) we then have
$$
[(a(x) - b(x)) (v_t(x,t))_x]_x = 0
$$
for every $(x,t)$ in $[0,L] \times [0,T]$. Integrating first with respect to $x$ using that $a(0)=b(0)$, we find that
$$
(a(x) - b(x)) (v_t(x,t))_x = 0.
$$
Since the coefficients are independent of time, we write this as
$$
[(a(x) - b(x)) (v_x(x,t))]_t = 0.
$$
Integrating with respect to time using $a(0)=b(0)$, and then putting $t = T$ we obtain
$$
(a(x) - b(x)) g_x(x) = 0
$$
for $0\leq x \leq L$. From the assumptions on $g$ we can conclude that $a(x)=b(x)$ also for
$0 < x < L$.

The final step in  the uniqueness argument is the proof of
unique identifiability of $f(x)$ in
\eqref{eq:model2}--\eqref{eq:ad-measurement}. Since $w=0$ and $a=b$, we have from (\ref{eq:model2.1_1}) that
$$
f_1(x) - f_2(x) = 0
$$
for every $x$ in $[0,L]$,
i.e. $f_1=f_2$ and the theorem is proved.
\end{proof}

We finish remarking that the main uniqueness argument is related to
the thermal diffusivity coefficient $a$. However, the arguments that
we have presented in this section appear not possible to generalize
to dimensions in space higher than one. This is due to the fact that
we would obtain an equation where the divergence of an element is
zero. However, we can not conclude that the given element would then
be a constant, as the simple example $(x,-y)$ shows.

\section{Conclusions and possible generalizations}

In this paper, we proved uniqueness for the inverse problem of
simultaneously identifying a spacewise heat conductivity and heat
source for a given final time measurement. The main result is based
on local Carleman estimates for parabolic problems
following~\cite{Yamamoto2009}. We did not strive for the most
general result, but only aimed at convincing the reader that there
can be at most one coefficient and source that satisfy a given final
time condition. Therefore, we assumed rather strong regularity on
the unknown parameters and worked with classical solutions. However,
these can be relaxed and it was indicated how to adjust the proof.
For conductivities and sources in spaces of integrable functions an
alternative proof of uniqueness in the one-dimensional case was
given, where many of the assumptions on smoothness were weakend.
Unfortunately, these arguments for the one-dimensional case can not
be generalized to dimensions higher than one. The one-dimensional
case does motivate the uniqueness of the inverse problem under less
regularity assumptions on the parameter spaces; to present such a
result in higher dimension is deferred to future work.

We remark that there are actually many strong results on Carleman
estimates for parabolic equations with very general assumptions on
the smoothness of the parameters \cite{Poisson2008,
Poisson_Carleman2008}. Moreover, we used local Carleman
estimates in order to simplify the proof of uniqueness in the higher-dimensional case. However, we may use global
Carleman estimates to avoid the assumption that the heat
conductivity is known close to the boundary of the body $\Omega$, see
\cite{KlibanovTimonov2004, Poisson_Carleman2008, Poisson2008,
Vessella2008, Yamamoto2009} for a general overview on the subject.

Therefore, the authors' conjecture is that there exists uniqueness of
the spacewise heat conductivity and heat source for a given final
time additional measurement, under weaker assumptions on the
parameter space. It will be addressed in a future works together with a regularization
method for the inverse problem.

\section*{Acknowledgments} A. De Cezaro is grateful for the support in the form
of a visiting fellowship obtained from the Brazil Visiting Fellows
Scheme at University of Birmingham, UK, offered in the summer of
2012 during which period this work were undertaken.

\appendix
\counterwithin{theo}{section} \numberwithin{equation}{section}
\section{Appendix}\label{appendix}

In this appendix, we show that the assumption (\ref{eq:assumption_Lemma_6.1}) on the final time
measurement \eqref{eq:ad-measurement} required in Section 2, can be satisfied, i.e. the
inverse problem \eqref{eq:model2}--\eqref{eq:ad-measurement} with (\ref{eq:assumption_Lemma_6.1}) imposed can have a solution for a
given initial data being sufficiently regular and with a localized heat
source.

The first result is on the regularity of the solution of
\eqref{eq:model2}. For simplicity,  we assume that all compatibility
conditions are satisfied and that $\Omega$ is an open, bounded and
regular subset of $\R^n$. We do not present a proof of the following
theorem, since it follows from standard regularity estimates for
parabolic equations. Moreover, since the coefficient and the source
are spacewise dependent, the solution is analytic in time. For details see,
for example, \cite{LM-vol1, LandSolUr-1968}.
\begin{theo}\label{th:regularity}
Let the coefficient and the source belong to the admissible set
$\mathcal{A}$. Then, for any $h \in H^k(\Omega)$, there exists a
unique solution $u\in C([0,T]; H^{k+1}(\Omega))\cap C((0,T);
H_0^1(\Omega) \cap H^{k+1}(\Omega)) \cap C^1((0,T);
H^{k+1}(\Omega))$ of the parabolic equation \eqref{eq:model2}.
\end{theo}
Note that the regularity above is far from optimal.
\begin{lemma}\label{lemma:assumptions1}
Let $k > \max\{3, n/2\}$. Then the assumption about the finiteness
of the constant $C$ in~\eqref{eq:carleman1} and assumption
\eqref{eq:smoothness-g} are satisfied for the solution of
\eqref{eq:model2} with additional data \eqref{eq:ad-measurement}.
\end{lemma}
\begin{proof}
This follows from Theorem~\ref{th:regularity} and the continuous
embedding of $H^k(\Omega)\cap H_0^1(\Omega)$ in
$C(\overline{\Omega})$, \cite{Ada75}.
\end{proof}
Now, we verify Assumption~\ref{eq:assumption_Lemma_6.1} in
Lemma~\ref{lemma:Yamamoto_6.1}. We remark that
Lemma~\ref{lemma:Yamamoto_6.1} is a corollary of \cite[Lemma
6.2]{Yamamoto2009} that reads as follows
\begin{pr}
Let the solution domain $\Omega$ be as above. Put $\zeta = a-b$. Consider the
first-order partial differential equation
\begin{align*}
(P_0 \zeta)(x) = \nabla \zeta(x) \cdot \nabla g(x) + \zeta(x)
\triangle g(x)
\end{align*}
where $g \in H^k(\Omega)$, for $k$ as in
Lemma~\ref{lemma:assumptions1}.

If the Carleman weight function $\varphi(x,T) \in
C^1(\overline{\Omega})$ satisfies
\begin{align}\label{eq:gradg}
\nabla g(x) \cdot \nabla \varphi(x,T) > 0 \,,\quad x \in
\overline{\Omega}\,,
\end{align}
then the Carleman estimate \eqref{eq:6.22} is satisfied.
\end{pr}
Therefore, we only need to prove that there exists a Carleman weight
function $\varphi$, a set of initial data and a heat source such
that \eqref{eq:gradg} is satisfied. However, by construction of the
Carleman weight function, we have
$$ \nabla \varphi(x,T) = \nabla d(x) e^{\lambda \beta(x,T)}\,.$$
Hence, is enough to guarantee that there exist $d(x)$ such that
\begin{align}\label{eq:gradg1}
\nabla g(x) \cdot \nabla d(x) > 0 \,,\quad x \in
\overline{\Omega}\,.
\end{align}
It is verified in the following lemma.
\begin{lemma}
Let $\varepsilon > 0$ and let $\mathcal{O}$ be any open set of
$\Omega$. There exists $d \in C(\overline{\Omega})$ with
$d|_{\partial \Omega} = 0$ and $|\nabla d(x)| \geq \varepsilon$ for
$x\in \Omega$, such that for any $h \in L^2(\Omega)$ there is a
sufficiently smooth source term $f$, having support in $\mathcal{O}$,
with
\begin{align}\label{eq:l-infty}
 \norm{\nabla u(f,h)(T) - \nabla d}_{L^\infty(\Omega)} <
\varepsilon/\sqrt{2}\,,
\end{align}
where $u(f,h)$ is the solution of \eqref{eq:model2} with source $f$
and initial data $h$.
Moreover, \eqref{eq:gradg1} holds.
\end{lemma}
\begin{proof}
Since we do not use global Carleman estimates, we can consider
$d$ to be zero on $\partial \Omega$. For the existence of such a function
$d$, see for example \cite[Section 2]{Poisson2008}. The density
argument then follows from \cite[Proposition 1.1]{Yuan_Yamamoto2009}, see
also \cite[Corollary 3.1]{Poisson2008}.

Hence, from $g(x) = u(f,h)(T)$, the regularity of $g(x)$ and $d(x)$
and the estimate~\eqref{eq:l-infty}, we have
$$
\varepsilon^2/2 > (ess\sup|\nabla g(x) - \nabla d(x)|)^2  \geq
|\nabla g(x) - \nabla d(x)|^2 = |\nabla g(x)|^2 - 2 \nabla g(x)
\cdot \nabla d(x) + |\nabla d(x)|^2 \,.$$
Therefore, since $|\nabla d| \geq \varepsilon$ it follows that
\begin{align*}
2 \nabla g(x) \cdot \nabla d(x) & > |\nabla g(x)|^2 + |\nabla
d(x)|^2 - \varepsilon^2/2   \geq |\nabla d(x)|^2 - \varepsilon^2/2
\geq \varepsilon^2/2\,.
\end{align*}
\end{proof}

The condition \eqref{eq:gradg1} can then be stated as
(\ref{eq:assumption_Lemma_6.1}) via a suitable transformation, see
the proof of \cite[Lemma 6.1]{Yamamoto2009}. Thus, we have shown
that there are conductivities and sources that can generate a final
time value such that (\ref{eq:assumption_Lemma_6.1}) holds, i.e. the
inverse problem \eqref{eq:model2}--\eqref{eq:ad-measurement} with
(\ref{eq:assumption_Lemma_6.1}) imposed will have a solution for
some data (and this solution is unique). Again, we have not
investigated optimal conditions and there might be other
restrictions on the final data that generate uniqueness. Note that
the condition (\ref{eq:assumption_Lemma_6.1}) is similar to the
condition imposed in the one- dimensional case in Section~3.

\bibliographystyle{amsplain}
\bibliography{uniqueness_tomas}

\end{document}